\documentclass[reqno]{amsart}
\usepackage{amssymb,graphicx}%
\usepackage{subfigure}
\usepackage{multicol,multirow}
\usepackage{amsmath,amsfonts}
\usepackage{mathrsfs}
\usepackage{amsthm}
\usepackage{color}
\usepackage{appendix}
\usepackage[numbers]{natbib}
\usepackage{ifpdf}
\ifpdf
 \usepackage[hyperindex]{hyperref}
\else
 \expandafter\ifx\csname dvipdfm\endcsname\relax
 \usepackage[hypertex,hyperindex]{hyperref}
 \else
 \usepackage[dvipdfm,hyperindex]{hyperref}
 \fi
\fi
\allowdisplaybreaks[4]
\theoremstyle{plain}
\newtheorem{thm}{Theorem}[section]

\theoremstyle{remark}
\newtheorem{rem}{Remark}[section]
\newtheorem{ex}{Example}[section]
\newtheorem{prob}{Problem}[section]
\numberwithin{equation}{section}

\begin{document}

\title{On a Fabric of Kissing Circles}

\author[V. \v{C}er\v{n}anov\'a]{Viera \v{C}er\v{n}anov\'a}
\address{Department of Mathematics and Computer Science, Faculty of Education, Trnava University, Priemyseln\'a 4, P. O. BOX 9, 918 43 Trnava, Slovakia}
\email{\href{mailto: V. Cernanova <vieracernanova@hotmail.com>}{vieracernanova@hotmail.com}, \href{mailto: V. Cernanova <viera.cernanova@truni.sk>}{viera.cernanova@truni.sk}}

\begin{abstract}
Applying circle inversion on a square grid filled with circles, we obtain a~configuration that we call a~fabric of kissing circles. In the paper, we focus on the curvature inside the individual components of the fabric, which are two orthogonal frames and two orthogonal families of chains. We show that the curvatures of the frame circles form a~doubly infinite arithmetic sequence (bi--sequence), whereas the curvatures in each chain are arranged in a~quadratic bi--sequence. We also prove a sufficient condition for the fabric to be integral.
\end{abstract}




\maketitle

\section{Introduction}
\label{sec:intro}
Classical geometric configurations of circles, such as the arbelos or the Pappus chain, and their numerous fascinating properties, have been known for about two thousand years.
At the beginning of the 19th century, the discovery of a~transformation called \textit{inversion in circle} (briefly \textit{inversion}) significantly simplified the solution of classical geometric problems. The last few decades have seen an expansion in the study of geometric configurations, namely circle packing problems, including the \textit{Apollonian packing}. The latter is a~fractal that begins with three externally tangent circles inscribed in one common circle, and is constructed by repeatedly inscribing circles into the triangular interstices in the configuration. Integral Apollonian  packings, that is, those in which the curvature of each circle is an integer, are of particular interest. 
The reader will find an engaging overview of the Apollonian  packing in \cite{Stephenson03}.  

In this paper we present a~novel configuration that we call a~\textit{fabric (of kissing circles)}, where the circles are organized in frames or in chains. 
Like the arbelos, the Pappus chain or the Apollonian packing, the fabric is closely related to the problem of Apollonius: given three generalized mutually touching circles with three points of tangency, there exist two other generalized circles touching the three.

The paper is organized as follows. In Section \ref{sec:fabric} the fabric is obtained by inverting a square grid filled with circles. Section \ref{sec:curvatures} brings our main results concerning the curvatures in the fabric, and a~sufficient condition for the fabric to be integral. To illustrate the power of the results from Section \ref{sec:curvatures}, we use them to solve two old sangaku problems in the last Section \ref{sec:sangaku}. 

We first introduced the fabric in \cite{cernan2020}, but our approach in Section \ref{sec:fabric} is not similar, and the relations between curvatures that we present in Section \ref{sec:curvatures} have not been published.
\section{Structure of the Fabric}
\label{sec:fabric}
Consider an infinite grid with evenly spaced horizontal and vertical lines, and a~circle inscribed in each square cell. Let  $A$ be an arbitrary point in the plane of the grid, let $\mathcal{C}$ be a~circle centered at $A.$ The inversion with respect to $\mathcal{C}$ transforms every element of this infinite  grid filled with circles into a~generalized circle. Throughout this paper, we call each of them a~{\it circle}, even though it may be a~straight line.
\subsection{Frames}
\label{sec:frames}
With respect to $\mathcal{C},$ the grid lines are inverted to two infinite families of circles passing through $A$: vertical lines into circles centered on a~horizontal line $p,$ and horizontal lines into circles centered on a~vertical line $p',$ both lines passing through $A.$ 
We denote the two families $\mathcal{V}$ and $\mathcal{H},$ respectively. 
Inversion preserves angles, thus $v \perp h$ for each pair of circles $(v,h)\in (\mathcal{V},\mathcal{H}).$ We will call $\mathcal{V}$ a~{\it vertical frame}, and $\mathcal{H}$ a~{\it horizontal frame}. Because the frames $\mathcal{V}$ and $\mathcal{H}$ are subsets of two orthogonal parabolic pencils with the common carrier $A,$ we will say that the vertical and horizontal frames are \textit{orthogonal}. 
\begin{figure}[bht]
	\centering
	\begin{minipage}{.5\textwidth}
		\centering
		\includegraphics[width=1.4in]{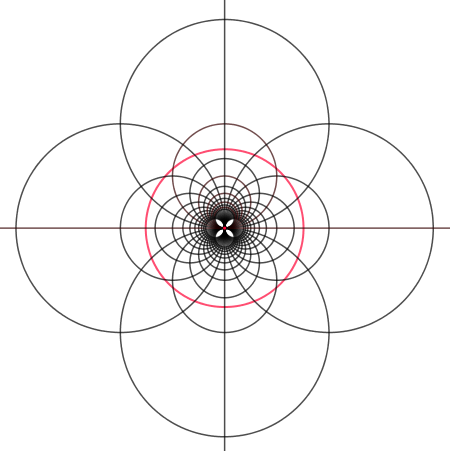}
		\caption{Orthogonal frames}{}
		\label{fig:frames2}
	\end{minipage}%
	\begin{minipage}{.5\textwidth}
		\centering
		\includegraphics[width=1.5in]{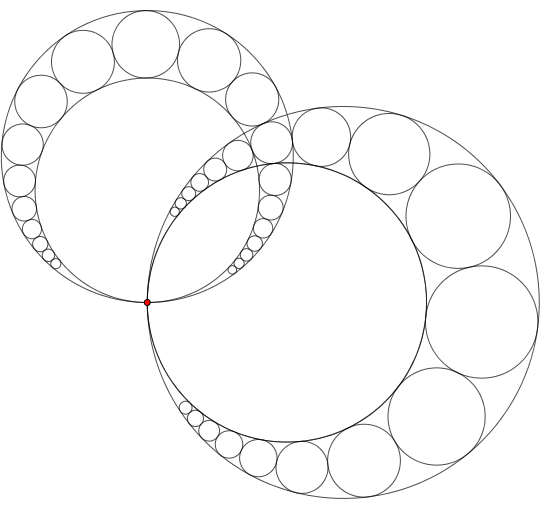}
		\caption{Orthogonal chains}{}
		\label{fig:orthochains4}
	\end{minipage}
\end{figure}
Figure~\ref{fig:frames2} represents a~pair of orthogonal frames, together with the reference circle and the carrier, which are red. The vertical line is a~generalized circle belonging to the vertical frame, similarly the horizontal line belongs to the horizontal frame.
\subsection{Chains}
Because a circle is inscribed in each cell of the grid, a~chain of tangent congruent circles is inscribed in each strip between two adjacent lines of the grid. The inversion with respect to $\mathcal{C}$ turns this chain into a~chain of circles,  each of which is tangent to its two neighbors in the chain and to two bounding circles. If the central line $p$ passes through the center of a~circle in the chain, a~part of the inverted chain is the Pappus chain.
This drives us to call the chain of circles obtained as above a~\textit{Pappus--like chain, PL--chain} or simply a~{\it chain}. Inverting a~column of circles gives a~{\it vertical chain}, inverting a~row of circles gives a~{\it horizontal chain}. 

Each pair of PL--chains, one of which is vertical and the other horizontal, shares a~circle. It is so because each row  of circles inscribed  in the grid shares a~circle with each column, and vice versa. In addition, the bounding circles of the two families of chains are orthogonal. We say that the chains are \textit{orthogonal} (Figure \ref{fig:orthochains4}).
\subsection{The Fabric}
We call a~{\it fabric of kissing circles} ({\it fabric}) the union of two orthogonal frames and two orthogonal families of inscribed chains , each of them obtained by the same circle inversion from a square grid filled with circles. 
The center $A$ of the reference circle is a~{\it carrier} of the fabric.

Several properties of the fabric copy those of the~grid filled with circles, which is a~direct consequence of the construction by inversion, between others:
\begin{itemize}
	\item[(i)] Each circle inscribed in a frame is shared by two orthogonal chains.
	\item[(ii)] Each region bounded by two frame circles contains a~chain. 
	\item[(iii)] The touch points of the circles within a chain lie on a~circle passing through $A.$ 
	\item[(iv)] Every fabric is invariant under reflexions in $4,2,1$ or $0$ axes that pass through the carrier $A,$ and under $4,2,1$ or $1$ rotations about $A,$ respectively. This number is determined only by the position in the grid of $A,$ which is also the center of the reference circle.
\end{itemize}
The set of reflexions and rotations corresponding to the fabric, endowed with the group operation that is composition of transformations, is symmetry group of the fabric. The group is:
\begin{itemize}
	\item[--] $D_4$, if $A$ is a vertex or the center of a~grid cell (Figure \ref{fig:8symC}),
	\item[--] $D_2$, if $A$ is the midpoint of a~side of a~grid cell  (will be shown in Figure \ref{fig:2sym}),
	\item[--] $D_1$, if $A$ lies on a side or a diagonal of a~grid cell, except the above (Figure \ref{fig:1symC}),
	\item[--] $C_1$, in all other cases (Figure \ref{fig:0symC}).
\end{itemize}
Figure \ref{fig:symmetries} shows three fabrics with different symmetry groups.
\begin{figure}[bht]
	\centering
	\subfigure[]{%
		\includegraphics[width=1.2in]{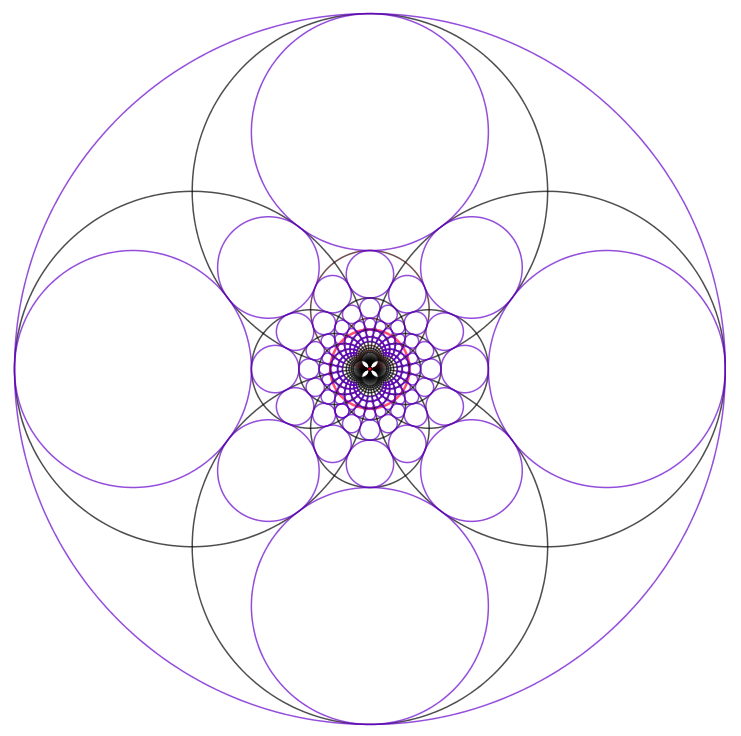}\hspace{2em}
		\label{fig:8symC}}
	\subfigure[]{%
		\includegraphics[width=1.2in]{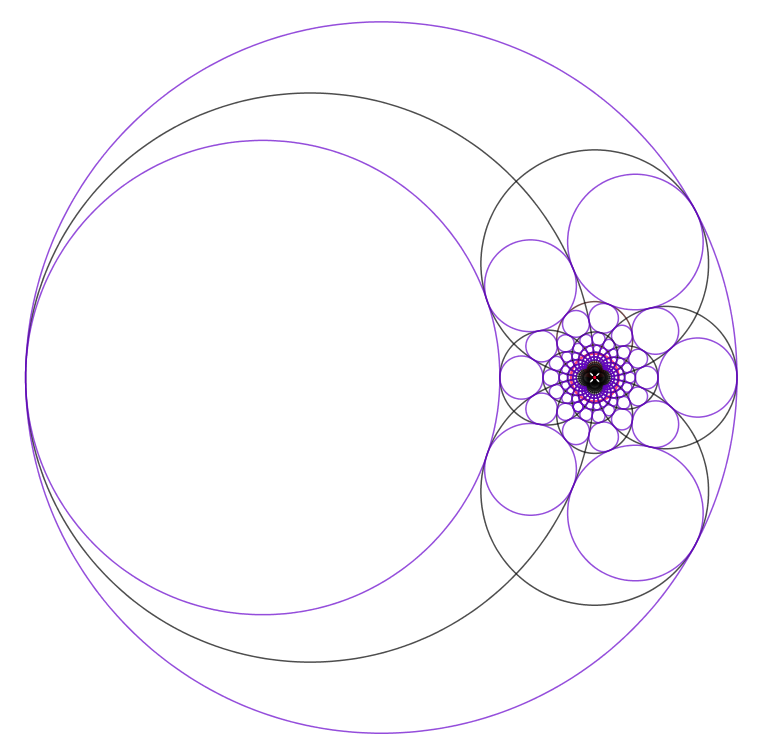}\hspace{2em}
		\label{fig:1symC}}
	\subfigure[]{%
		\includegraphics[width=1.2in]{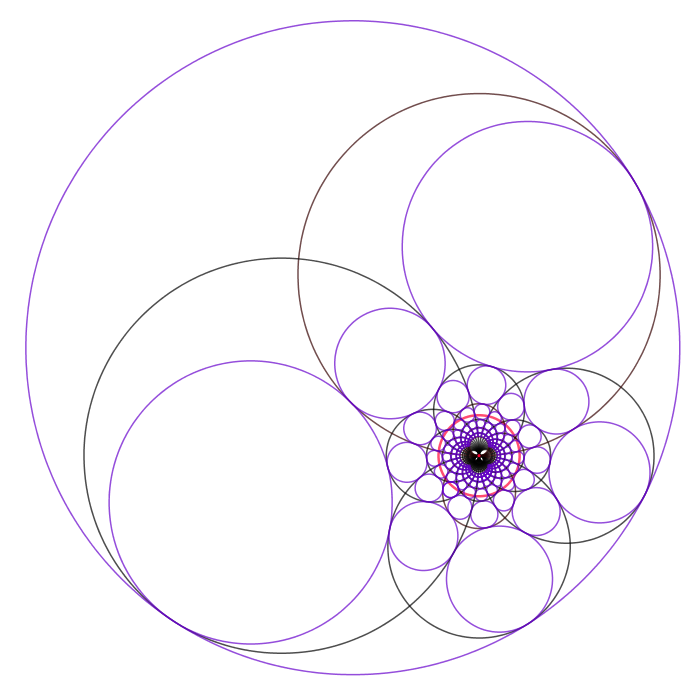}
		\label{fig:0symC}}
	\caption{From the left: fabrics with symmetry group $D_4, D_1,C_1$}
	\label{fig:symmetries}
\end{figure}
For the convenience of the reader, the frames are black, the chains are purple, and the reference circle is red.
Notice that the~frame circles pass through the carrier, whereas the chain circles do not.
\section{Curvatures in the Fabric}
\label{sec:curvatures}
Unlike the~grid filled with circles, the fabric exhibits visually endless variability of shapes, many of which are very pleasing to the eye. 
This visual variability of the fabric is numerically manifested  by the curvature of its components. 
Theorems \eqref{th:F-cur} and \eqref{th:QPC}, which are the main results in this paper, reveal the basic relations between the curvatures in the frame and in the chain, respectively. These relations apply to every fabric, regardless of its appearance or symmetry.  
\subsection{Curvatures in the frame}
\begin{thm}
	\label{th:F-cur}
	In a~plane, let $\mathcal{C}$ be a~circle of radius $r$ centered at $A.$ Let $s, t$ be two adjacent parallel lines in a~square grid, let $s', t'$ be the inverses of $s$ and $t$ with respect to $\mathcal{C}.$ Then the curvatures of $s',t'$ differ, in absolute value, by a~constant independent of the choice of $s$ and $t.$
\end{thm}
\begin{figure}[bht]
	\centering
	\includegraphics[width=2.1in]{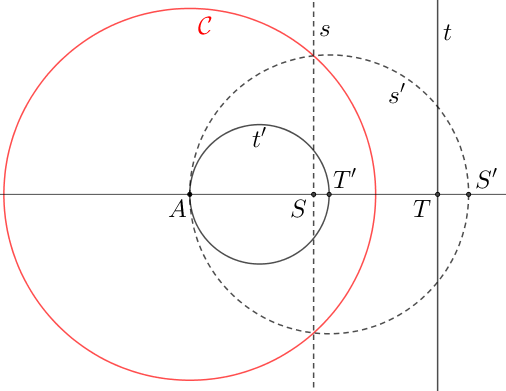}
	\caption{Reference circle (red), adjacent grid lines and their inverses}{}
	\label{fig:delta}
\end{figure} 
\begin{proof}
	Trivially, $s',t'$ are generalized circles tangent at $A.$ The line passing through $A$ and perpendicular to $s$ and $t$ meets $s,t,s',t'$ at $S,T,S',T',$ respectively, all of which are not necessarily different. Denote with $\overline{XY}$ the distance between points $X,Y.$ By definition of circle inversion
	$$\overline{AS}\cdot \overline{AS'}= \overline{AT}\cdot \overline{AT'}=r^2.$$
	\begin{enumerate}
		\item Assume that  the~situation is as in Figure \ref{fig:delta}.	If $\kappa_{s'}, \kappa_{t'}$ are the curvatures of $s', t',$ respectively, then
		\begin{equation*}
		\label{Delta}
		|\kappa_{t'}-\kappa_{s'}|=\bigg|\frac{2}{\overline{AT'}}-\frac{2}{\overline{AS'}}\bigg|=\bigg|\frac{2\overline{AT}-2\overline{AS}}{r^2}\bigg|=\frac{2d}{r^2},
		\end{equation*} 
		which is given only by the spacing $d=\overline{ST}$ between grid lines and the radius $r$ of the reference circle $\mathcal{C}.$ 
		\item If $A$ lies on the open segment $ST,$ it lies also on $S'T'.$ The circles $s'$ and $t'$ belong to opposite sides of the frame with respect to the carrier $A.$ The signs of the curvatures $\kappa_{s'}$ and $\kappa_{t'}$ are opposite, cf. Remark \ref{rem1}. Again, $|\kappa_{t'}-\kappa_{s'}|={2d}/{r^2},$ as
		$$|\kappa_{t'}-\kappa_{s'}|=\bigg|\frac{2\overline{AT}+2\overline{AS}}{r^2}\bigg|=\frac{2d}{r^2}.$$
		\item If one of the lines $s, t$ passes through $A,$ then still 
		$|\kappa_{t'}-\kappa_{s'}|={2d}/{r^2}$ because either $\overline{AS}=0, s'=s$ and $\kappa_{s'}=0,$ or $\overline{AT}=0, t'=t$ and $\kappa_{t'}=0.$
	\end{enumerate}	
\end{proof}
	According to Theorem \ref{th:F-cur}, the curvatures of frame circles are arranged in an~arithmetic bi--sequence. In addition, Theorem \ref{th:F-cur} implicitly states that the common difference is the same in both frames, and is independent of the position of $A$ in the grid. We will label the common difference 
$$	\Delta=|\kappa_{t'}-\kappa_{s'}|=\frac{2d}{r^2}.$$
\begin{rem}
	\label{rem1}
	The curvatures of circles located on opposite sides of the parabolic pencil with respect to the carrier have opposite signs. This complies the phenomenon known in the optics: the sign of the curvature depends on the position of the vertex (in our case, the point $A$) relative to the center (it is the center of the considered circle). Without loss of generality, we will agree that these signs are the same as in the coordinate system, in which $A$ is the origin and central lines $p$ and $p'$ of the frames are coordinate axes.
\end{rem}	
\subsection{Curvatures in the chain}
To prove Theorem \ref{th:QPC} we will use relation \eqref{id:desc} that connects radii of four circles touching  each other externally as in Figure \ref{fig:DCT0}, and which is known from a~1643 letter written by {Ren\'e} Descartes \cite{RD}. 
\begin{thm}[Descartes Circle Theorem]
	\label{th:desc}
	If four circles with curvatures $k_1,k_2,k_3,k_4$ are tangent to each other at six distinct points, then 
	\begin{equation}\label{id:desc}
	\bigg(\sum_{i=1}^{4}k_i\bigg)^2=2\sum_{i=1}^{4}k_i^2.
	\end{equation}
\end{thm}
According to Apollonius, if the curvatures $k_1, k_2, k_3$ of three mutually tangent generalized circles with three points of tangency are known, then the quadratic equation \eqref{id:desc} has two real roots, say $k, k'$ (double roots are counted twice), which are the curvatures of two generalized circles touching the three. From \eqref{id:desc} we have
\begin{equation}
\label{id:desc1}
k+k'=2(k_1+k_2+k_3).
\end{equation}
Relation \eqref{id:desc} also applies to all configurations of four two--by--two tangent generalized circles (called {\it Descartes configurations} in \cite{lagarias_beyond}), in which no three share a~point of tangency, see Figure \ref{fig:descartes_config}.
Curvatures meet the rules:
\begin{itemize}
	\item[(i)] \label{config1} If three circles are internally tangent  to the fourth as in Figure \ref{fig:DCT1}, the curvature of the outermost circle enters \eqref{id:desc} with a  sign opposite to the three, which is poetically explained in \cite{soddy1936}. 
	\item[(ii)] The curvature of a~straight line is zero.
\end{itemize}
\begin{figure}[bht]
	\centering
	\subfigure[]{%
		\includegraphics[width=1.1in]{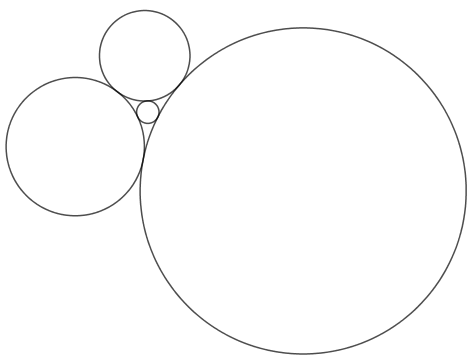}\hspace{1.5em}
		\label{fig:DCT0}}
	\subfigure[]{%
		\includegraphics[width=0.8in]{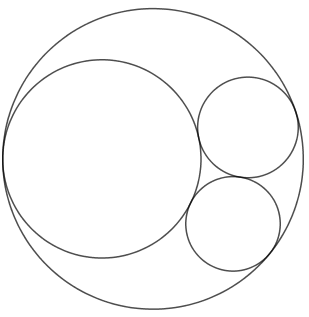}\hspace{1.5em}
		\label{fig:DCT1}}
	\subfigure[]{%
		\includegraphics[width=1.0in]{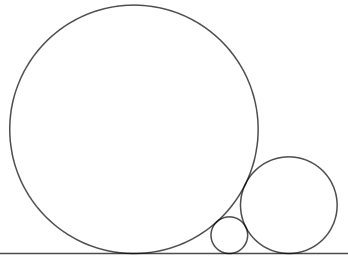}\hspace{2em}
		\label{fig:DCT4}}
	\subfigure[]{%
		\includegraphics[width=0.42in]{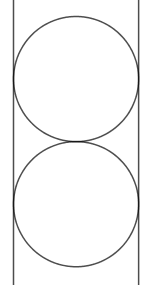}
		\label{fig:DCT3}}
	\caption{Descartes configurations}
	\label{fig:descartes_config}
\end{figure}

Let us now look at the curvatures in a~chain.
\begin{thm}
	\label{th:QPC}
	Let $\gamma_0, \gamma_1$ be two touching circles inscribed in a  region bounded by two adjacent frame circles $s',t'$ in a~fabric, let  $\left\{\gamma_{n}\right\}_{n=-\infty}^{\infty}$ be a chain generated by the four circles. Let us label with $\kappa_0, \kappa_1, \kappa_{s'}, \kappa_{t'}$ the respective curvatures of $\gamma_0, \gamma_1, s', t'.$ Then the curvature $\kappa_n$ of $\gamma_n$ satisfies 
	\begin{equation}
	\label{quadr}
	\kappa_n=\kappa_0+D\cdot n+\Delta \cdot n(n-1), n\in \mathbb Z,
	\end{equation}
	where $D=\kappa_1-\kappa_0$ and $\Delta=|\kappa_{t'}-\kappa_{s'}|.$	
\end{thm}
\begin{proof}
	Each circle $\gamma_n$ is tangent to $s'$ and $t',$ which are internally tangent at $A,$ therefore the Descartes theorem \ref{th:desc} can be applied. Due to (\ref{id:desc1}) where one of $\kappa_{s'}, \kappa_{t'}$ is zero or enters with a~negative sign, the curvatures $\kappa_{n+1}, \kappa_{n-1}$ of two chain circles  adjacent to $\gamma_n$ satisfy
	$$\kappa_{n+1}+\kappa_{n-1}=2(\kappa_n+\Delta),$$
	or equivalently 
	$$\kappa_{n+1}-\kappa_n=\kappa_n-\kappa_{n-1}+2\Delta.$$
	Therefore, $\{\kappa_n\}_{n=-\infty}^{\infty}$ is a~quadratic bi--sequence with a~recurrence relation 
	$$\kappa_{n+1}=\kappa_n+D+2n\Delta, n\in \mathbb{Z}, $$
	and the closed--form expression \eqref{quadr} of the general term. 
\end{proof}
\begin{rem}
	The curvature of each circle $\gamma_n,$ as obtained by  \eqref{quadr}, is independent of the choice of $\gamma_0, \gamma_1$ because the quadratic bi--sequence of curvatures is doubly infinite.
\end{rem}
\subsection{Integral fabric}
We will say that the~fabric is {\it integral}, if the curvature of each single circle -- either in the~frame or in the fabric~chain -- is an integer. In Theorem \ref{integral} we state a~sufficient condition for the~fabric to be integral.
\begin{thm}
	\label{integral}
	In a~fabric $\mathcal{F},$ let $\gamma_0, \gamma_1$ be a~pair of touching circles inscribed in the region bounded by frame circles $s',t'.$ Label with $\Psi$ the chain generated by $s',t',\gamma_0, \gamma_1.$ Let $\alpha$ and $\beta$ be arbitrary neighbors of $\gamma_0$ and $\gamma_1,$ respectively, in the chains  orthogonal to $\Psi.$ Denote the two chains $\Psi'_{0}, \Psi'_{1}.$
	If the curvatures of $s', t', \gamma_0, \gamma_1, \alpha$ and $\beta$ are integers, then the fabric $\mathcal{F}$ is integral.
\end{thm}
\begin{proof} 
	As the curvatures of two adjacent frame circles are integers, by Theorem \ref{th:F-cur} it is the same for each circle in the vertical or horizontal frame.
	
	According to Theorem \ref{th:QPC} all the circles belonging to $\Psi'_{0}$ or $\Psi'_{1}$ have integral curvatures. Consequently, each circle in an arbitrary chain orthogonal to $\Psi'_{0}$ or $\Psi'_{1}$ has an integral curvature. Because every chain circle belongs to a~chain orthogonal to the~given chain, the proof is complete.
\end{proof}
Theorem \ref{integral} states that six circles with integer curvature suitably distributed in the fabric are sufficient for the fabric to be integral. Recall that, unlike the fabric, the Apollonian packing is integral if its generating quad -- the configuration of four circles as in Figure \ref{fig:DCT1} -- is integral. 
\begin{ex}
	\label{ex2}
	The fabric with symmetry group $D_2$ pictured in Figure \ref{fig:2sym}, where chain circles are purple and frame circles black, is integral. We will derive all the curvatures in the fabric; several circles are already labeled with their curvature. Figure \ref{fig:invertedCircle} shows the corresponding grid filled with circles, in which each line and each circle is labeled with the curvature of its inverse shape in the fabric. 
	The reference circle and its center are red in both figures. 
\end{ex}	
\begin{figure}[bht]
	\centering
	\includegraphics[width=2.6in]{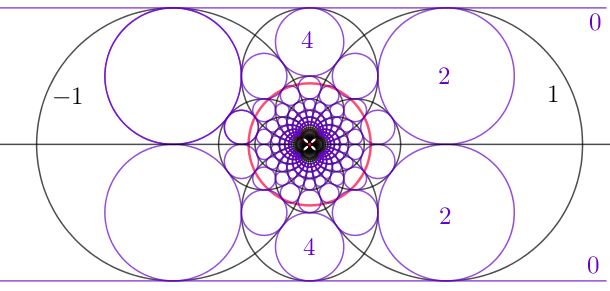}
	\caption{A fabric with symmetry group $D_2$}{}
	\label{fig:2sym}
\end{figure}
Assume that the distance between two extreme horizontal lines in Figure \ref{fig:2sym} is $2.$  We can easily determine the smallest curvatures.

The black central line passing through the carrier (curvature $0$) is the largest generalized circle in the horizontal frame. The two  frame circles adjacent to the line are congruent (radius 1/2) and tangent from outside. Therefore their curvatures $\pm 2$ have opposite signs. We find $\Delta = |2-0|=2.$ The arithmetic bi--sequence of curvatures in this frame is $\{\dots, -6,-4,-2,0,2,4,6,\dots\}.$ 	
These are the labels of horizontal lines in Figure \ref{fig:invertedCircle}.
\begin{figure}[bht]
	\centering
	\includegraphics[width=4in]{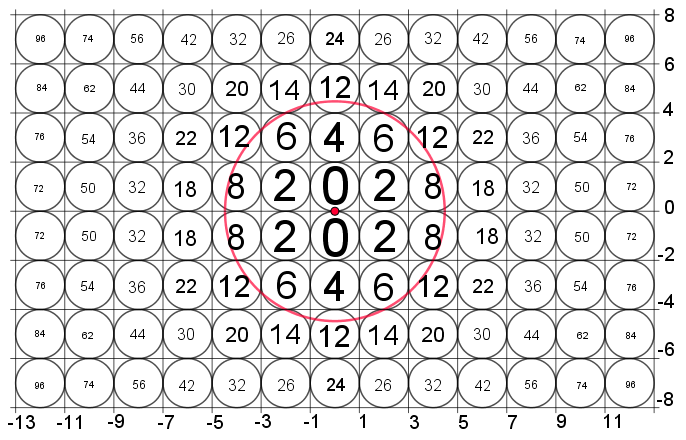}
	\caption{A grid labeled with curvatures relative to the red reference circle}{}
	\label{fig:invertedCircle}
\end{figure}

Now look at the vertical frame. The two largest black circles are congruent (radius $1$) and touch each other from the outside, so their curvatures $\pm 1$ are opposite. As they are adjacent frame circles, the value $\Delta=2$ is confirmed. We notice that the curvatures in the horizontal frame are all even integers, whereas in the vertical frame they are all odd integers. The latter are the labels of the vertical lines in Figure \ref{fig:invertedCircle}.

Most chains in the fabric are easy to identify. Consider the vertical chain inscribed in the rightmost large circle. Two congruent touching circles (curvature $2$) belonging to the chain give $D=|2-2|=0.$  With the use of \eqref{quadr} we find the bi--sequence of curvatures $\{\dots,26,14,6,2,2,6,14,26,\dots\};$ see Figure \ref{fig:invertedCircle}, the labels in a~column.

The strangest chain in Figure \ref{fig:2sym} is coming out of the carrier upwards, includes the upper horizontal line followed by the~lower one, and the chain continues in smaller and smaller circles upwards, back to the carrier. This is a~vertical chain inscribed in the region ``between" two large congruent circles, members of the vertical frame. Because the two lines are neighbors in the chain, $D=|0-0|=0.$ The bi--sequence of curvatures $\{\dots, 12,4,0,0,4,12,\dots\}$ appears in the~middle column in Figure \ref{fig:invertedCircle}.
\section{Two Sangaku Problems}
\label{sec:sangaku}
Solving sangaku problems is growing in popularity. In Japanese tradition, mathematical problems drawn on wooded tablets (sangaku) originally hung in shrines or temples. For a~rich and recognized source of information on Japanese {\it sacred mathematics} with many problems and solutions, see \cite{FukPed,FukRot}. 
Bellow we use Theorem~\ref{th:QPC} to solve two sangaku problems with chains of circles. Problem \ref{gumma} dated 1814 from the Gumma prefecture  is available in \cite{FukPed} as Problem 1.8.6.
\begin{prob}[Figure \ref{fig:prob1}] 
	\label{gumma}
	A chain of tangent circles is inscribed in a~region bounded by two internally tangent semi--circles and a common central line. If the radii of the circles are $r_1>r_2> \dots,$ prove that 
	\begin{equation}
	\label{eq:prob1}
	\frac{7}{r_4}=\frac{2}{r_7}+\frac{5}{r_1}.
	\end{equation}	
\end{prob} 
\noindent {\it Solution.}
The chain is part of a~PL--chain with two largest congruent circles symmetric about the horizontal axis, as shows Figure \ref{fig:prob1sol}. Equal curvatures of congruent neighbors give $D=0.$ 
We set $\kappa_i=1/r_i$ and rewrite \eqref{eq:prob1} as
$$7\kappa_4-(2\kappa_7+5\kappa_1)=0.$$
With the use of \eqref{quadr} we obtain
$$7\kappa_4-(2\kappa_7+5\kappa_1)=7(\kappa_0+4D+12\Delta)-2(\kappa_0+7D+42\Delta)-5(\kappa_0+D)=9D=0,$$ which proves \eqref{eq:prob1}.
\begin{figure}[bht]
	\centering
	\subfigure[]{%
		\includegraphics[width=1.9in]{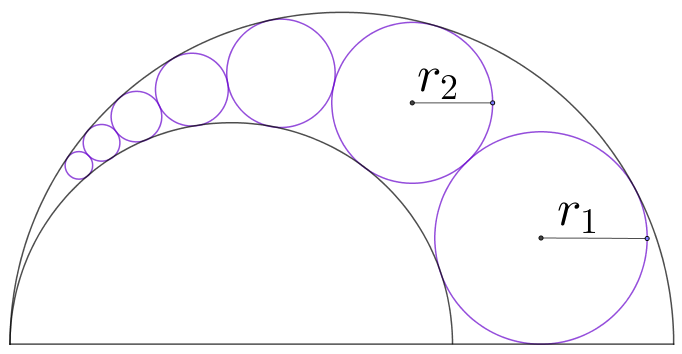}\hspace{5.5em}
		\label{fig:prob1}}
	\subfigure[]{%
		\includegraphics[width=0.95in]{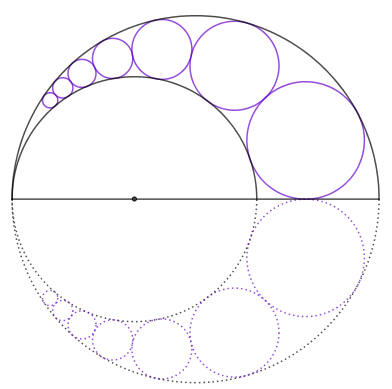}
		\label{fig:prob1sol}}
	\caption{A chain of seven circles}
	\label{fig:gumma}
\end{figure}
\begin{rem}
	We observe that:
	\begin{enumerate}
		\item The radii of the  semi--circles in Figure \ref{fig:prob1} do not affect the result because $\Delta$ and $\kappa_0$ drop out from $7\kappa_4-(2\kappa_7+5\kappa_1).$
		\item Identity \eqref{eq:prob1} applies exclusively to PL--chains that contain a~pair of touching congruent circles.
	\end{enumerate}
\end{rem} 
Unlike the Problem \ref{gumma}, a $3:2$ ratio of the radii of bounding circles is crucial in Problem \ref{menuma}. In \cite{FukRot}, p. 312, we read: {\it The problem is from a lost tablet hung by Kanei Teisuke in 1828 in the Menuma temple of Kumagaya city, Saitama prefecture. We know of it from the 1830 manuscript Saishi Shinzan or Collection of Sangaku by Nakamura Tokikazu (?–1880). (Japan Academy.)}
\begin{prob}[Figure \ref{fig:prob2}]
	\label{menuma}
	Show that $r_7=r/7.$
\end{prob}
\noindent {\it Solution.}
The configuration is symmetric about the~central line of bounding circles and of three congruent circles. The inscribed chain is also symmetric about this line, as shows Figure \ref{fig:prob2sol}. Curvatures of the outermost and the inner bounding (frame) circles are $a=1/(3r)$ and $b=1/(2r),$ hence $\Delta=|b-a|=1/(6r).$
As the largest chain circle with radius $r$ is considered as $N^o 1$ in the chain, we have $\kappa_1=1/r.$
Due to the symmetry, the neighbors of this circle are congruent: $\kappa_0=\kappa_2.$ The relation  \eqref{id:desc1} becomes 
$$
2\kappa_0=2(-a+b+\kappa_1),$$ which gives 
$D=\kappa_1-\kappa_0=-\Delta=-{1}/{(6r)}.$ From this we get $\kappa_0=\kappa_1-D=7/(6r).$
With the use of \eqref{quadr} we find 
$$\kappa_7=\kappa_0+7D+42\Delta=\frac{7}{6r}-\frac{7}{6r}+\frac{42}{6r}=\frac{7}{r},$$ or equivalently $r_7=r/7.$ 
\begin{figure}[bht]
	\centering
	\subfigure[]{%
		\includegraphics[width=1.5in]{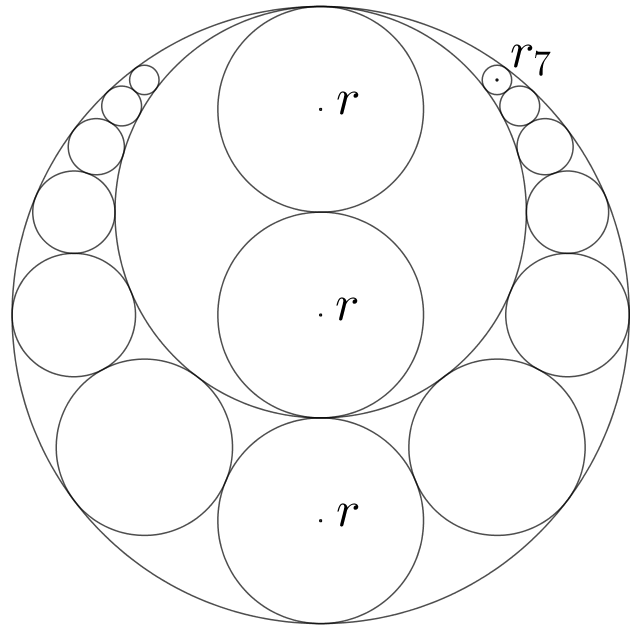}\hspace{5.5em}
		\label{fig:prob2}}
	\subfigure[]{%
		\includegraphics[width=1.55in]{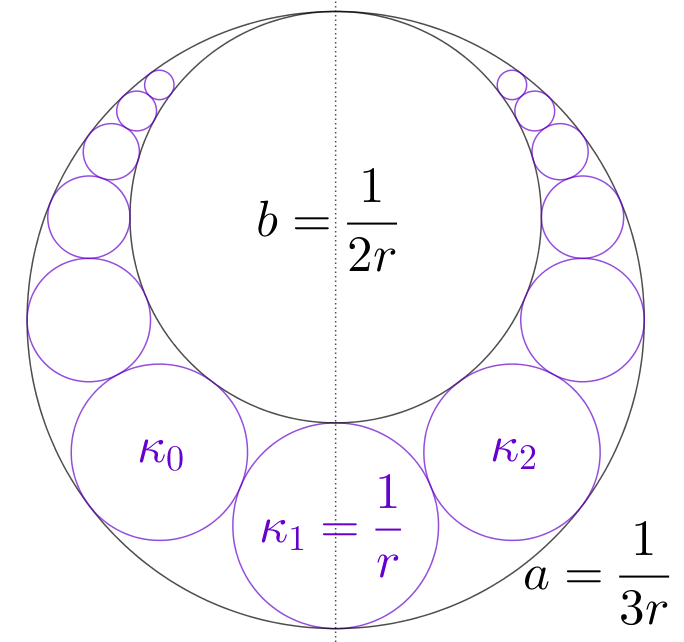}
		\label{fig:prob2sol}}
	\caption{A seventh circle problem}
	\label{fig:saitama}
\end{figure}


\begin{thebibliography}{9}
\bibitem{cernan2020}
	{V.\ \v Cer\v nanov\'a}, 
	{On a Configuration Resulting from Circle Inversions}.
	{\it Proceedings, 19th Conference on Applied Mathematics APLIMAT 2020} (Spektrum STU, 2020), 238--243. 

\bibitem{RD}	
	{Descartes correspondence avec Elisabeth}. 
	(Egmond du Hoef, Nov. 1643).\\  
	\url{https://web.archive.org/web/20050306191834/http://www.ac-nice.fr:80/philo/textes/Descartes-Elisabeth/Descartes-Elisabeth.htm}

\bibitem{FukPed}
	{H.\ Fukagawa, D.\ Pedoe},
	{\it Japanese Temple Geometry Problems. San Gaku}.
	(The Charles Babbage Research Centre, Winnipeg, Canada, 1989). 

\bibitem{FukRot}
	{H.\ Fukagawa, T.\ Rothman}, 
	{\it Sacred Mathematics: Japanese Temple Geometry}. 
	(Princeton University Press, 2008).

\bibitem{lagarias_beyond}
	{J. C.\ Lagarias, C. L.\ Mallows, A. R.\ Wilks},
	{Beyond the Descartes Circle Theorem}.
	{\it Amer. Math. Monthly} {\bf 109}(4) (2002), 338--361.\\
	\url{http://dx.doi.org/10.1080/00029890.2002.11920896}
	
\bibitem{soddy1936}
	{F.\ Soddy},  
	{The Kiss Precise}. {\it Nature} {\bf 137} (3477) (1936), p. 1021.\\ \url{https://doi.org/10.1038/1371021a0}
	
\bibitem{Stephenson03}
	{K.\ Stephenson},
	{Circle packing: A mathematical tale}.
	{\it Notices Amer. Math. Soc.} \textbf{50}(11) (2003), 1376--1388.\\  \url{http://www.ams.org/notices/200311/fea-stephenson.pdf}
\end{thebibliography}
\end{document}